\documentclass[reqno,oneside,12pt]{amsart}
\usepackage{
  amsmath,
  amssymb,
  amsthm,
  geometry,
  paralist,
  tikz,
  tikz-cd,
  verbatim,
  relsize, 
  bbm, 
  dsfont,
  mathtools,
  xifthen, 
  mathrsfs, 
  calligra, 
  enumitem,
}

\usepackage[parfill]{parskip}   
\DeclareMathOperator{\rep}{rep}
\DeclareMathOperator{\GL}{GL}

\usepackage[T1]{fontenc}
\usepackage[all,cmtip]{xy}
\usepackage{caption}
\usepackage{subcaption}
\usetikzlibrary{angles}

\renewcommand\labelenumi{(\roman{enumi})}
\renewcommand\theenumi\labelenumi

\newcommand{\unit}{\mathds{1}}

\RequirePackage{ifpdf}
\ifpdf
  \IfFileExists{pdfsync.sty}{\RequirePackage{pdfsync}}{}
  \RequirePackage[pdftex,
   colorlinks = true,
   urlcolor = myblue, 
   citecolor = mygreen, 
   linkcolor = myred, 
   bookmarksopen=true,
   unicode, psdextra]{hyperref}
\else
  \RequirePackage[hypertex, unicode, psdextra]{hyperref}
\fi
\hypersetup{
    colorlinks=true,
    linkcolor=blue,
    filecolor=magenta,      
    urlcolor=cyan,
    citecolor = blue
} 

\RequirePackage{ae, aecompl, aeguill} 

\usepackage{quiver}

\usepackage{cleveref} 

\usepackage{adjustbox}

\tikzcdset{
	arrows = crossing over
}

\usepackage[maxbibnames=99,
			maxalphanames=10, 
			style=alphabetic,
            url=false]{biblatex}
\usepackage{xurl}

\usepackage{dynkin-diagrams}

\usepackage[boxsize=.5em]{ytableau}

\usetikzlibrary{arrows.meta} 

\makeatletter
\providecommand\@dotsep{5}
\makeatother

\newcommand{\<}{\langle}
\renewcommand{\>}{\rangle}

\DeclareMathOperator{\id}{id}

\DeclareMathOperator{\Hom}{Hom}
\DeclareMathOperator{\iHom}{\mathscr{H}\text{\kern -3pt {\calligra\large om}}\,}
\DeclareMathOperator{\rad}{rad}

\DeclareMathOperator{\iEnd}{\mathscr{E}\text{\kern -3pt {\calligra\large nd}}\,}
\DeclareMathOperator{\Ext}{Ext}
\DeclareMathOperator{\iExt}{\mathscr{E}\text{\kern -3pt {\calligra\large xt}}\,\,}

\DeclareMathOperator{\proj}{\mathsf{proj}}
\newcommand{\Cproj}[1]{\ifthenelse{\isempty{#1}}{\cC\text{-}\proj}{#1\text{-}\proj}}

\DeclareMathOperator{\inj}{\mathsf{inj}}
\newcommand{\Cinj}[1]{\ifthenelse{\isempty{#1}}{\cC\text{-}\inj}{#1\text{-}\inj}}

\newcommand{\lmod}[1]{{#1}\text{-}\mod}

\newcommand{\leftdual}[1]{#1^*}
\newcommand{\rightdual}[1]{{}^*\!#1}

\makeatletter
\let\@tl\triangleleft
\let\@tr\triangleright
\newcommand*{\@smallTriangle}[2]{\vcenter{\hbox{\scalebox{0.75}{\ensuremath{#1#2}}}}}
\newcommand*{\@medTriangle}[2]{\vcenter{\hbox{\scalebox{0.90}{\ensuremath{#1#2}}}}}
\newcommand*{\lact}{\mathbin{\mathpalette\@smallTriangle\@tr}}
\newcommand*{\ract}{\mathbin{\mathpalette\@smallTriangle\@tl}}
\let\@btl\blacktriangleleft
\let\@btr\blacktriangleright
\newcommand*{\blact}{\mathbin{\mathpalette\@medTriangle\@btr}}
\newcommand*{\bract}{\mathbin{\mathpalette\@medTriangle\@btl}}
\renewcommand*{\triangleright}{\lact}
\renewcommand*{\triangleleft}{\ract}
\renewcommand*{\blacktriangleright}{\blact}
\renewcommand*{\blacktriangleleft}{\bract}

\newcommand{\cP}{\mathcal{P}}

\newcommand{\cC}{\mathcal{C}}
\newcommand{\cF}{\mathcal{F}}
\newcommand{\cM}{\mathcal{M}}
\newcommand{\cN}{\mathcal{N}}
\newcommand{\cA}{\mathcal{A}}

\renewcommand{\vec}{\mathsf{vec}}

\renewcommand{\mod}{\mathsf{mod}}

\numberwithin{equation}{section}
\newtheorem{Theorem}[equation]{Theorem}
\newtheorem{Proposition}[equation]{Proposition}
\newtheorem{Lemma}[equation]{Lemma}

\theoremstyle{definition}
\newtheorem{Remark}[equation]{Remark}
\newtheorem{Example}[equation]{Example}

\newtheorem{eg}[equation]{Example}
\newenvironment{example}[1][]{\begin{eg}[#1] \pushQED{\qed}}{\popQED \end{eg}}
\newtheorem{Definition}[equation]{Definition}

\crefname{lemma}{Lemma}{Lemma}
  \crefname{corollary}{Corollary}{Corollary}
  \crefname{theorem}{Theorem}{Theorem}
  \crefname{definition}{Definition}{Definition}
   \crefname{proposition}{Proposition}{Proposition}
\crefname{question}{Question}{Question}
 \crefname{section}{Section}{Section} 
   \crefname{construction}{Construction}{Construction}
   \crefname{generalization}{Generalization}{Generalization}
  \crefname{construction}{Construction}{Construction}
  \crefname{notation}{Notation}{Notation}
   \crefname{example}{Example}{Example}
  \crefname{remark}{Remark}{Remark}
  \crefname{fact}{Fact}{Fact}
  \crefname{conjecture}{Conjecture}{Conjecture}
  \crefname{motivation}{Motivation}{Motivation}  
  \crefname{figure}{Figure}{Figure}  
  \crefname{assumption}{Assumption}{Assumption}

\title[Rigid and $\tau$-locally free modules]{Rigid and $\tau$-locally free modules via Taft algebras}
\author{Edmund Heng}
\address{E.~Heng: School of Mathematics and Statistics, University of Sydney, Australia}
\email{edmund.heng@sydney.edu.au}

\author{Calvin Pfeifer}
\address{C.~Pfeifer: Mathematical Institute, University of Cologne, Weyertal 86-90, 50931 Köln, Germany}
\email{cpfeifer@math.uni-koeln.de}

\date{}

\begin{document}

\maketitle

\begin{abstract}
    We show that rigid modules are $\tau$-locally free for a large family of 1-Iwanaga--Gorenstein algebras. 
    This family includes many cases of the algebras introduced by Geiss--Leclerc--Schr\"oer in their {\it Inventiones} paper ``Quivers with relations for symmetrizable Cartan matrices I: Foundations'', hence answering a question posed by Geiss--Leclerc--Schr\"oer for many new cases.
    Our result is achieved by introducing coactions of Taft algebras and applying the recent theory developed by the first-named author and Coulembier.
\end{abstract}

\tableofcontents

\section{Introduction}\label{sec:intro}

Throughout, we fix an algebraically closed field $\Bbbk = \overline{\Bbbk}$; soon we will impose also the existence of certain primitive roots of unity in $\Bbbk$.
In this work, we will be concerned with a class of algebras associated to triples $(Q,R,n)$, where
\begin{enumerate}
    \item $Q = (Q_0,Q_1,s,t)$ is a finite acyclic quiver;
    \item $R = (R_0,R_1)$ is a subquiver\footnote{By a \emph{subquiver} of a quiver $Q$ we mean a tuple $R=(R_0,R_1)$ with $R_0\subseteq Q_0$ and $R_1\subseteq Q_1$ such that $s(\alpha),t(\alpha) \in R_0$ for all $\alpha\in R_1$.} of $Q$; and
    \item $n\in \mathbb{N}$ is a positive integer.
\end{enumerate}
We denote by $Q^\varepsilon$ the quiver obtained from $Q$ by adding a loop $\varepsilon_i$ to each vertex $i \in Q_0$. 
We write $\Bbbk Q^{\varepsilon}$ for the path algebra of the quiver $Q^{\varepsilon}$.
A basis of $\Bbbk Q^\varepsilon$ is given by all paths in $Q^{\varepsilon}$, and multiplication in $\Bbbk Q^\varepsilon$ is defined on basis elements by concatenation of paths if possible and $0$ otherwise.
We follow the convention to read and concatenate paths from right to left, the same way we compose our maps.

\begin{Definition}\label{def:H(Q,R,n)}
    Let $(Q,R,n)$ be a triple as above. 
    We define the factor algebra
    \[H(Q,R,n) \coloneqq \dot{H}(Q,R_0,n) / I(Q,R_1)\]
    where 
    \[\dot{H}(Q,R_0,n) \coloneqq \Bbbk Q^\varepsilon / \< \varepsilon_i, \varepsilon_j^n \mid i\in Q_0\setminus R_0, j\in R_0\>\]
    and the ideal $I(Q,R_1)\subseteq \dot{H}(Q,R_0,n)$ is generated by the commutativity relations $$\varepsilon_{t(\alpha)}\alpha = \alpha\varepsilon_{s(\alpha)}$$ for all $\alpha \in R_1$.
\end{Definition}

Note that $\dot{H}(Q,R_0,n) = H(Q,R',n)$ with $R'_0 = R_0$ and $R'_1 = \emptyset$.

Most importantly, our algebras $H(Q,R,n)$ include some of the algebras $H(C,D,\Omega)$ introduced by Geiss--Leclerc--Schr\"oer in \cite{GLS_quiversI}
which are associated to a generalised Cartan matrix $C$, a symmetriser $D$, and an acyclic orientation $\Omega$.
More precisely, $H(C,D,\Omega)$ is isomorphic to $H(Q,R,n)$ for some triple $(Q,R,n)$ if and only if the symmetriser $D=\textup{diag}(c_i)$ satisfies $c_i \in \{1,n\}$ for all $i$.
This includes many of the finite and affine Dynkin cases. Namely, all those $H(C,D,\Omega)$ with
\begin{enumerate}
    \item $C$ of simply-laced finite or affine Dynkin type and $D$ any symmetriser; or
    \item $C$ of non-simply-laced finite or affine Dynkin type and $D$ a minimal symmetriser, except those of type $\mathbb{BC}_k$ for $k\ge 2$.
\end{enumerate}
In turn, $H(Q,R,n)$ is isomorphic to $H(C,D,\Omega)$ for some triple $(C,D,\Omega)$ if and only if $R$ is a full\footnote{We say that a subquiver $R$ of $Q$ is \emph{full} if $\alpha\in R_1$ for all $\alpha\in Q_1$ with $s(\alpha),t(\alpha)\in R_0$.} subquiver of $Q$.
Like Geiss--Leclerc--Schr\"oer's algebras, all our algebras are 1-Iwanaga--Gorenstein and isomorphic to tensor algebras of species of algebras in the sense of \cite{Kulshammer_prospecies} (see Section \ref{sec:species}).

To state our main theorem, note that a module $M$ over $H(Q,R,n)$ is the same data as giving for each $i\in Q_0$ a module $M_i$ over a truncated polynomial algebra 
$$H_i = 
\begin{cases}
    B(n) \coloneqq \Bbbk[\varepsilon]/\<\varepsilon^{n}\> & \text{if $i\in R_0$;} \\
    \Bbbk & \text{otherwise;}
\end{cases}$$
and for each arrow $\alpha\in Q_1$ a linear map $M_\alpha\colon M_{s(\alpha)} \to M_{t(\alpha)}$ which is required to be a homomorphism of $B(n)$-modules whenever $\alpha\in R_1$. 
Following \cite{GLS_quiversI}, we say that a module $M$ is \emph{locally free} if $M_i$ is a free module over $H_i$ for all $i \in Q_0$; we say that $M$ is \emph{$\tau$-locally free} if $\tau^k(M)$ is locally free for all $k\in\mathbb{Z}$ ($\tau$ denotes the Auslander--Reiten translate).
Our main theorem is the following
\begin{Theorem}[=Theorem~\ref{thm:mainthm}]
    Any rigid\footnote{This means that $\Ext^1(M,M)=0$; not to be confused with ``rigid'' often used in the context of monoidal/tensor categories; see Remark \ref{rmk:rigid}.} module over $H(Q,R,n)$ is $\tau$-locally free.
\end{Theorem}

This answers -- for those $H(C,D,\Omega)$ which belong to our family $H(Q,R,n)$ -- the question posed by Geiss--Leclerc--Schr\"oer in \cite[after Corollary~6.3]{GLS_rigidSchur}. 
Their question is motivated as part of their extensive program to extend many of the spectacular connections of Lie theory beyond the representation theory of hereditary algebras, which allow them to work over algebraically closed fields.
An instance is a bijective correspondence between real Schur roots (defined using non-crossing partitions or the Cartan lattice) of the Kac--Moody algebra associated to the Cartan matrix $C$ and isomorphism classes of indecomposable rigid modules over $H(C,D,\Omega)$ that are locally free \cite[Theorem 1.2(a)]{GLS_rigidSchur}.
For hereditary algebras, a bijection between real Schur roots and rigid modules was known (see \cite{CB_ex-seq-quiver, Ringel_ex-seq-Artin, HuberyKrause}), though the non-symmetric types would require working over non-algebraically closed fields (that depend on the Cartan matrix).
This raises the natural question whether one can drop the `locally free' assumption on rigid modules in Geiss--Leclerc--Schr\"oer's correspondence above to resemble the hereditary situation even more.

A lemma due to Demonet \cite[Lemma 6.2]{GLS_rigidSchur} shows that all $\tau$-rigid\footnote{This means that $\Hom(M,\tau(M)) = 0$. In particular, every $\tau$-rigid module is rigid by the Auslander--Reiten formula.} modules over $H(C,D,\Omega)$ are locally free. 
Demonet's proof also works for our algebras $H(Q,R,n)$, but does not apply to rigid modules. 
We note that once one shows that rigid modules over $H(C,D,\Omega)$ are locally free, the fact that they are $\tau$-locally free follows from \cite[Proposition 11.4]{GLS_quiversI}.
Nonetheless, we emphasise that our proof is independent of this result, and applies to all algebras $H(Q,R,n)$; we will briefly discuss our proof strategy in the next subsection.
Finally, a geometric argument (over an algebraically closed field $\Bbbk$ of characteristic $0$) shows that the question of Geiss--Leclerc--Schr\"oer has an affirmative answer for all algebras $H(C,D,\Omega)$ with symmetriser $D = \textup{diag}(c_i)$ such that $\gcd(c_i,c_j) = 1$ for all $i,j\in Q_0$; see \cite[Corollary 3.16]{Pfeifer}. 
We note that the same geometric argument applies to the algebras $\dot{H}(Q,R_0,n)$.
As far as we know, for all other cases our theorem is new.

\subsection{Proof overview}
Our main observation is that $\cM \coloneqq \lmod{H(Q,R,n)}$ comes equipped with a natural action of the finite tensor category $\cC_n \coloneqq \lmod{T_n(q)}$, where $T_n(q)$ is a finite dimensional Hopf algebra known as the Taft algebra (Definition \ref{defn:Taft}). In other words, there is a natural way to equip $\cM$ with the structure of a $\cC_n$-module category (Definition~\ref{defn:modulecat}), which allows us to apply the theory developed by the first-named author and Coulembier in \cite{CH_eqAuscor-findim}.

For any $\cC$-module category $\cM$, the authors in loc.\ cit.\ introduce the subcategory $\cP\cM \subseteq \cM$ consisting of all modules that are direct summands of modules $P\lact M$ for $M\in \cM$ and $P \in \proj(\cC)$.
It was shown that this subcategory satisfies numerous remarkable properties, one of which is being closed under the Auslander--Reiten translates $\tau^k$ for all $k\in \mathbb{Z}$.

In our setting, it follows from the definition of the $\cC_n$-action that all modules in $\cP\cM$ are locally free -- hence the general result above implies that they are all $\tau$-locally free.
Thus it suffices to show that given a rigid module $M$, $P\lact M$ contains $M$ as a summand for some\footnote{In fact, $\cP\cM$ is functorially finite and $P\lact M$ can be chosen to be a right or left $\cP\cM$-approximation of $M$.} projective $P \in \proj(\cC_n)$ -- this is shown in Theorem~\ref{thm:mainthm}.
We note that this hinges on a particular property of the $\cC_n$-action on $\cM$ (see Lemma~\ref{lem:orbit-vanishing-Ext-splits} and \ref{lem:simple_action_on_rigids}), and it is not true for general $\cC$-module categories that $\cP\cM$ contains all rigid modules; see Example \ref{eg:rigidnotinPM}.
Nonetheless, we hope that our result can be generalised to all algebras $H(C,D,\Omega)$ considered by Geiss--Leclerc--Schr\"oer in \cite{GLS_quiversI} via other appropriate actions of finite tensor categories $\cC$.

\subsection*{Notation and conventions}
An algebra will always mean an algebra over $\Bbbk$.
Given an algebra $A$, we will use $\lmod A$ to denote the category of finitely generated \emph{left} $A$-modules.
We say that an abelian category $\cA$ is \emph{finite} if $\cA \cong \lmod A$ for some finite dimensional algebra $A$.
Given an abelian category $\cA$, we use $\proj(\cA)$ to denote its category of projective objects.

\subsection*{Acknowledgments}
EH is immensely grateful for the invitation to visit Cologne by Sibylle Schroll, which resulted in this collaboration. We are also deeply grateful to Sibylle Schroll for many insightful discussions.
EH was supported by the Australian Research Council grant DP230100654 and the visit was partially supported by the ``Overseas Conference and Research Visit Support Program'' offered by the Faculty of Science, University of Sydney.

\subsection*{Statement on AI usage} The authors declare that no artificial intelligence was used in this work.

\section{Preliminaries: tensor categories and their module categories}
\subsection{Tensor categories from Hopf algebras}
\begin{Definition}[\protect{\cite[Definition 4.1.1 \& Chapter 6]{EGNO}}]
    A \emph{finite tensor category} $\cC$ is a finite abelian category $\cC$ equipped with a monoidal structure $-\otimes -: \cC \times \cC \to \cC$ such that the monoidal unit $\unit$ is simple and every object $X \in \cC$ has a left dual $\leftdual{X}$ and right dual $\rightdual{X}$.
\end{Definition}
\begin{Remark}\label{rmk:rigid}
    We warn the reader that in the theory of monoidal categories, the word ``rigid'' is commonly used to mean an object in a monoidal category that has both left and right duals \cite[Definition 2.10.11]{EGNO} -- this is \emph{not} the same as a rigid module (in the sense that $\Ext^1(M,M)=0$).
\end{Remark}
We note that in a finite tensor category $\cC$, both left and right duals of a projective object remains projective, hence all projectives are injectives, and vice versa \cite[Proposition 6.1.3]{EGNO}.
In other words, $\cC$ is a quasi-Frobenius category, i.e.\ $\cC$ is equivalent to the category of modules over a finite dimensional self-injective algebra.

The finite tensor categories of interest in this paper will come from finite dimensional Hopf algebras, whose definition we now recall:
\begin{Definition}[\protect{\cite[Definition 5.3.10]{EGNO}}]
    A \emph{bialgebra} $R$ is a vector space $R$ that is simultaneously equipped with the structure of an algebra $(R,\mu:R \otimes_\Bbbk R \to R, \iota:\Bbbk \to R)$ and a coalgebra $(R,\Delta: R \to R\otimes_\Bbbk R, \epsilon: R \to \Bbbk)$, such that the comultiplication $\Delta$ and counit $\epsilon$ are both maps of algebras.
    A \emph{Hopf algebra} is a bialgebra $R$ that comes with an invertible linear map $S: R \to R$ known as the \emph{antipode}, satisfying
    \[
    \mu \circ(\id \otimes S)\circ \Delta = i \circ \epsilon = \mu \circ(S \otimes \id)\circ \Delta.
    \]
\end{Definition}
Note that it follows from the condition above that $S:R \to R$ is an antihomomorphism of algebras (and coalgebras).
Moreover, the antipode of a bialgebra is uniquely determined if it exists \cite[Proposition 5.3.5]{EGNO}.

\begin{Definition}
    Let $R$ be a finite dimensional Hopf algebra. Then $\lmod R$ can be equipped with a monoidal structure defined as follows. Given two $R$-modules $X$ and $Y$, we define $X \otimes Y$ to be the $R$-module with underlying object given by $X \otimes_\Bbbk Y$ and $R$-module structure defined by 
    \[
    r\cdot (x\otimes y) = \sum_{i} (r_{1,i}\cdot x) \otimes (r_{2,i}\cdot y)
    \]
    where $\Delta(r) = \sum_i r_{1,i}\otimes r_{2,i}$.
    There is a distinguished simple object $\unit \coloneqq \Bbbk$ with $R$-module structure induced by counit $\epsilon:R \to \Bbbk$, which is the monoidal unit of $\lmod R$.
    Moreover, $\lmod R$ is a finite tensor category with respect to this monoidal structure \cite[Corollary 5.3.7]{EGNO}. Given an $R$-module $X$, its left dual $\leftdual{X}$ and right dual $\rightdual{X}$ both have underlying object given by the $\Bbbk$-linear dual $D(X)$, with $R$-module structure defined respectively by:
    \[
    (a\cdot f)(x) = f(S(a)\cdot x); \qquad (a\cdot g)(x) =g(S^{-1}(a)\cdot x).
    \]
\end{Definition}

\subsection{Module categories from Hopf-comodule algebras}
\begin{Definition}[\protect{\cite[Definition 7.3.1]{EGNO}}]\label{defn:modulecat}
    A (left) \emph{module category} $\cM$ over a finite tensor category $\cC$, or a (left) \emph{$\cC$-module category}, is a finite abelian category $\cM$ equipped with a bifunctor $-\lact - : \cC \times \cM \to \cM$ that is exact in both variables\footnote{Exactness in the second variable is implied by existence of (tensor) duals in $\cC$; the functor $X \lact -: \cM \to \cM$ has left adjoint $\leftdual{X}\lact -$ and right adjoint $\rightdual{X}\lact -$, hence must be exact.}, satisfying the associativity condition and unitality condition (see \cite[Definition 7.1.1]{EGNO}).
\end{Definition}

We will be mainly interested in module categories over $\cC = \lmod R$ for a finite dimensional Hopf algebra $R$, which can be described through the following notion.
\begin{Definition}[\protect{\cite[\S1.5]{AM_modulecatHopf}}]
    Let $R$ be a finite dimensional Hopf algebra and set $\cC = \lmod R$.
    A finite dimensional algebra $A$ is said to be a (left) \emph{$R$-comodule algebra} if it is an $R$-comodule whose comodule structure map $\lambda:A \to R\otimes_\Bbbk A$ is a map of algebras.
    Given a finite dimensional $R$-comodule algebra $A$, $\lmod A$ is naturally a $\cC$-module category: for $X \in \cC$ and $M \in \lmod A$, we define $X \lact M$ to be the $A$-module with underlying object $X \otimes_\Bbbk M$ and $A$-module structure given by
    \begin{equation}\label{eq:modulecatHopf}
        a \cdot (x \otimes m) = \sum_i (a_{0,i}\cdot x) \otimes (a_{1,i}\cdot m),
    \end{equation}
    where $\lambda(a) = \sum_i a_{0,i} \otimes a_{1,i}$.
\end{Definition}

\begin{Remark}\label{rmk:two-natural-comodule-algebra}
    Given a Hopf algebra $R$, there are (at least) two natural $R$-comodule algebras.
    The first is $R$ itself, using the comultiplication $\Delta: R \to R\otimes_\Bbbk R$ as $\lambda$; the other is $\Bbbk$, using instead the unit map $\iota: \Bbbk \to R \cong R\otimes_\Bbbk \Bbbk$ as $\lambda$.
    The corresponding module category structure (over $\lmod R$) on $\vec_\Bbbk=\lmod \Bbbk$ obtained is also the module category structure induced by the forget functor $\lmod R \to \vec_\Bbbk$ (which is monoidal and exact, also commonly called a \emph{fibre functor}).
\end{Remark}

\subsection{The precluster tilting subcategory \texorpdfstring{$\cP\cM$}{PM}}
Let $\cM$ be a module category over a finite tensor category $\cC$.
The following subcategory of $\cM$ introduced in \cite{CH_eqAuscor-findim} will play a crucial role in this paper.
\begin{Definition}\label{defn:PM}
    For $\cM$ a $\cC$-module category, we define $\cP\cM$ to be its smallest additive subcategory closed under taking direct summands containing $P\lact M$ for all $P \in \proj(\cC)$ and all $M \in \cM$.
\end{Definition}

It was shown in \cite{CH_eqAuscor-findim} that $\cP\cM$ satisfies numerous remarkable properties; the property that is relevant for us is the following:
\begin{Theorem}[\protect{\cite[Theorem 1.1(i)]{CH_eqAuscor-findim}}]\label{thm:ARtranslate-closed}
    $\cP\cM$ is closed under all Auslander--Reiten translates: for $M\in \cP\cM$, $\tau^k(M) \in \cP\cM$ for all $k \in \mathbb{Z}$.
\end{Theorem}

\section{The algebras $H(Q,R,n)$}\label{sec:species}

Recall the Definition \ref{def:H(Q,R,n)} of our algebras $H(Q,R,n)$ from the introduction.
It should be apparent to the experts that these algebras are tensor algebras of \emph{species of algebras} in the sense of \cite{Kulshammer_prospecies}.
In this section we will briefly recall this construction to introduce some notation, to provide a more conceptual perspective on this class of algebras, and to mention some of their fundamental properties.

For the remainder of this paper, we fix $(Q,R,n)$ and simply denote $H\coloneqq H(Q,R,n)$.
Abbreviate $B(n) := \Bbbk[\varepsilon]/\langle \varepsilon^n \rangle$, 
and for all $i\in Q_0$ and $\alpha\in Q_1$ set
\begin{align*}
    H_i := 
    \begin{cases}
        B(n) & \text{if $i\in R_0$}; \\
        \Bbbk & \text{else;}
    \end{cases}
    &&
    H_\alpha := 
    \begin{cases}
        H_{t(\alpha)} \otimes_{B(n)} H_{s(\alpha)} & \text{if $\alpha\in R_1$;}\\
        H_{t(\alpha)} \otimes_{\Bbbk} H_{s(\alpha)} & \text{else;}
    \end{cases}
\end{align*}
where we consider $H_\alpha$ as an $H_{t(\alpha)}$-$H_{s(\alpha)}$-bimodule in the natural way.
To simplify notation, we also set
\begin{align*}
    B_0 := \prod_{i\in Q_0} H_i, &&
    B_1 := \bigoplus_{\alpha\in Q_1} H_\alpha,
\end{align*}
and consider $B_1$ as a $B_0$-$B_0$-bimodule via the natural projections.
The associated \emph{tensor algebra} is then defined as 
\begin{align*}
    T(Q,R,n) := B_0 \oplus \bigoplus_{\ell \geq 1}
    (\underbrace{B_1 \otimes_{B_0} \cdots \otimes_{B_0} B_1}_{\text{$\ell$ copies of $B_1$}})
\end{align*}
with multiplication given by concatenation of elementary tensors; see \cite[Definition 2.9]{Kulshammer_prospecies} for more details.

\begin{Lemma}
    There is an isomorphism of algebras $H \cong T(Q,R,n)$.
\end{Lemma}
\begin{proof}
    An isomorphism is given by sending the idempotent $e_i$ of $H$ (resp. loop $\varepsilon_i$) at a vertex $i\in Q_0$ to the $i$th copy of $1$ (resp. $\varepsilon$) in $B_0$,
    and an arrow $\alpha\in Q_1$ to $1\otimes 1$ at position $\alpha$ in $B_1$.
    It is straight forward to check that this indeed determines a well-defined isomorphism of algebras. 
\end{proof}

It will be more convenient to use the description of $H$-modules given in the next remark to prove Proposition \ref{prop:Taftaction} below. 
\begin{Remark}\label{rem:H-mod}
    The category of left $H$-modules is equivalent to the category of representations of the species of algebras given by $(H_k,H_\alpha)_{k\in Q_0,\alpha\in Q_1}$; see \cite[Proposition 2.11]{Kulshammer_prospecies}.
    In more basic terms, which is also evident from the description of $H$ as a factor algebra of a path algebra, an $H$-module $M$ is determined by the following data:
    \begin{enumerate}
        \item for each $i\in Q_0$ a vector space $M_i$, which carries the structure of a $B(n)$-module for all $i\in R_0$;
        \item for each $\alpha\in Q_1$ a linear map $M_\alpha\colon M_{s(\alpha)} \to M_{t(\alpha)}$, which is required to be a homomorphism of $B(n)$-modules whenever $\alpha\in R_1$.\footnote{Technically, in a representation of a (pro-)species of algebras we have $M_\alpha: H_\alpha\otimes_{H_{s(\alpha)}} M_{s(\alpha)} \to M_{t(\alpha)}$ for each $\alpha \in Q_1$. We leave it to the reader to check that this is the same as our simplified description (using appropriate adjunctions between extension of scalars and forgetful functors).}
    \end{enumerate}
    Furthermore, a homomorphism $f\colon M\to N$ of $H$-modules is given by a linear map $f_i\colon M_i \to N_i$ for each $i\in Q_0$ such that 
    \begin{enumerate}
        \item $f_i$ is a homomorphism of $B(n)$-modules for all $i\in R_0$;
        \item $N_\alpha f_{s(\alpha)} = f_{t(\alpha)} M_\alpha$ for all $\alpha\in Q_1$.
    \end{enumerate}
\end{Remark}

As an immediate application of \cite[Proposition 3.5]{Kulshammer_prospecies} we get:

\begin{Proposition}\label{prop:1IG}
    The algebra $H := H(Q,R,n)$ is $1$-Iwanaga--Gorenstein.
    Furthermore, the following are equivalent for a module $M\in\lmod{H}$:
    \begin{enumerate}
        \item The projective dimension of $M$ as an $H$-module is at most $1$.
        \item For every $i\in Q_0$, the module $M_i$ is free as an $H_i$-module.
    \end{enumerate}
\end{Proposition}

\section{Coactions of Taft algebras}

Recall that we fixed $(Q,R,n)$ and $H\coloneqq H(Q,R,n)$.
From here on, we assume that $\Bbbk$ contains all primitive $n$-th roots of unity.
We first introduce the relevant finite tensor category, which we will later show acts on $\lmod H$.

\begin{Definition}[Taft algebras]\label{defn:Taft}
    Let $n \in \mathbb{N}$ and $q\in \Bbbk$ be a primitive $n$-th root of unity.
    The \emph{Taft algebra} $T_n(q) = \Bbbk\<g,x\>/\<g^n=1, x^n=0, gx=qxg\>$ is the Hopf algebra with coalgebra structure given by
    \[
    \Delta(g) = g\otimes g, \quad \Delta(x)=x\otimes 1 + g\otimes x, \qquad \epsilon(g)=1, \quad \epsilon(x)=0.
    \]
    Its (invertible) antipode is defined by $S(g)=g^{-1}$ and $S(x)=-g^{-1}x$. 
    This makes $\lmod{T_n(q)}$ into a finite tensor category, which we denote by $\cC_n$.
\end{Definition}

Let us describe some of the salient properties of $\cC_n$.
A complete set of pairwise orthogonal primitive idempotents of $T_n(q)$ is given by 
\begin{equation*}
    e_\ell \coloneqq \frac{1}{n}\sum_{j=0}^{n-1}q^{\ell j}g^j
\end{equation*}
for $0 \leq \ell < n$.
We denote the corresponding indecomposable projective and simple modules by 
\[
    P_\ell \coloneqq T_n(q) e_\ell \quad \text{and} \quad 
    S_\ell \coloneqq P_\ell/\rad(P_\ell),
\]
for all $0\leq \ell < n$,
respectively.
Note that $S_\ell = \Bbbk \bar{e}_\ell$ is $1$-dimensional,
where $\bar{e}_\ell \in S_\ell$ denotes the residue class of $e_\ell$ in $S_\ell$,
and the left $T_n(q)$-action is given by
\begin{align}\label{eq:simpleTaftmodule}
    x \cdot \bar{e}_\ell = 0
    && \text{and} &&
    g \cdot \bar{e}_\ell = q^{-\ell} \bar{e}_\ell.
\end{align}
One readily sees that 
\[
S_\ell \otimes S_{\ell'} \cong S_{\ell+\ell'}
\]
with subscripts $0\leq \ell,\ell' < n$ taken modulo $n$.
In particular, $S_0$ is the monoidal unit of $\cC_n$, which agrees with the distinguished simple module defined by the counit $\epsilon: T_2(q) \to \Bbbk$.

\subsection{Coactions of Taft algebras on truncated polynomial algebras}

To describe the coaction of Taft algebras on truncated polynomial algebras, we use the following notation. 
Let $B$ be an algebra, and $\sigma\colon B\to B$ an algebra automorphism. 
Denote by ${_\sigma}B$ the $B$-$B$-bimodule with ${_\sigma}B = B$ as a right $B$-module and left $B$-module structure given by $a\cdot b = \sigma(a)b$ for all $a\in B$ and $b\in {_\sigma}B$.

\begin{Lemma}\label{lem:truncatedpoly-exact}
    Let $B(n) \coloneqq  \Bbbk[\varepsilon]/\<\varepsilon^n\>$.
    Then $B(n)$ is a $T_n(q)$-comodule algebra via the comodule structure:
    \[
    \lambda(\varepsilon)=x\otimes 1 + g\otimes\varepsilon,
    \]
    so that $\lmod{B(n)}$ is a $\cC_n$-module category.
    \begin{enumerate}
        \item \label{item:P-action-truncatedpoly} 
        Let $M\in \lmod{B(n)}$.
        For every $0\leq \ell <n$, we have that $$P_\ell \lact M \cong B(n)^{\oplus \dim_\Bbbk(M)}$$ is a free $B(n)$-module.
        \item \label{item:S-action-truncatedpoly}
        Denote by $\sigma\colon B(n) \to B(n)$ the algebra automorphism given by $\sigma(\varepsilon) \coloneqq q^{-1} \varepsilon$.
        For each $0\leq \ell < n$, there is a natural isomorphism 
        \begin{align*}
            S_\ell\lact - \simeq {_{\sigma^\ell}}B(n) \otimes_{B(n)} -
        \end{align*}
        of endofunctors on $\lmod{B(n)}$.
    \end{enumerate}
\end{Lemma}
\begin{proof}
    Note that there is an embedding of algebras $B(n) \to T_n(q)$ defined by $\varepsilon \mapsto x$; the $T_n(q)$-comodule structure $\lambda$ on $B(n)$ is therefore simply the restriction of $\Delta$ to $B(n)$.
    It then follows from general theory that $\lmod{B(n)}$ is a $\cC_n$-module category.

    For \ref{item:P-action-truncatedpoly}, 
    one can calculate directly that $P_\ell \lact S \cong B(n)$ for the unique simple $B(n)$-module $S$; we leave this calculation to the reader.
    Now given $M \in \lmod{B(n)}$, consider the Jordan--H\"older filtration of $M$, whose simple subquotients are necessarily isomorphic to $S$.
    Applying the exact functor $P_\ell \lact -$ to this filtration, every short exact sequence in the resulting filtration has to split because $P_\ell\lact S$ is a projective $B(n)$-module.
    It then follows that $P_\ell \lact M \cong B(n)^{\oplus \dim_\Bbbk(M)}$, noting that the Jordan--H\"older length of $M$ equals its $\Bbbk$-dimension because $\dim_\Bbbk (S) = 1$.

    For \ref{item:S-action-truncatedpoly}, 
    let $M\in\lmod{B(n)}$ again.
    We have isomorphisms of vector spaces
    \begin{align*}
        \varphi_M\colon M &\xrightarrow{\sim} S_\ell \lact M; &&& 
        \psi_M\colon M &\xrightarrow{\sim} {_{\sigma^\ell}}B(n) \otimes_{B(n)} M; \\
        m &\mapsto \bar{e}_\ell \otimes m; &&&
        m &\mapsto 1 \otimes m.
    \end{align*}
    We claim that a natural isomorphism is given by 
    \begin{align*}
        \eta_M := \psi_M\varphi_M^{-1}\colon S_\ell \lact M \xrightarrow{\sim} {_{\sigma^\ell}}B(n) \otimes_{B(n)} M.
    \end{align*}
    The map $\eta_M$ is clearly natural in $M\in\lmod{B(n)}$.
    That $\eta_M$ is a homomorphism of modules over $B(n)$ follows from the following direct calculation using (\ref{eq:modulecatHopf}) and (\ref{eq:simpleTaftmodule}):
    The left actions of $B(n)$ on $S_\ell\lact M$ and ${_{\sigma^\ell}}B(n) \otimes_{B(n)} M$, respectively, are given by 
    \begin{align*}
        \varepsilon \cdot (\bar{e}_\ell \otimes m) 
        &= (x \cdot \bar{e}_\ell) \otimes (1 \cdot m) + (g \cdot \bar{e}_\ell) \otimes (\varepsilon \cdot m) 
        = q^{-\ell}\, \bar{e}_\ell \otimes (\varepsilon \cdot m); \\
        \varepsilon \cdot (1 \otimes m) 
        &= (\varepsilon \cdot 1) \otimes m 
        = \sigma^\ell(\varepsilon) \otimes m
        = q^{-\ell}\, \varepsilon \otimes m
        = q^{-\ell}\, 1 \otimes (\varepsilon \cdot m).
    \end{align*}
\end{proof}

\begin{Remark}
    The lemma above is most likely well-known among the experts of Hopf algebras.
    In fact, exact\footnote{A $\cC$-module category $\cN$ is said to be \emph{exact} if $P\lact N\in \proj(\cN)$ for all $P \in \proj(\cC)$.} $\cC_n$-module categories have been classified in \cite[Theorem 4.10]{EO_FTC} and \cite[\S 8.1]{Mombelli_uqsl2}, where $\lmod{B(n)}$ is part of the classification.
\end{Remark}

\subsection{Coactions of Taft algebras on $H(Q,R,n)$}

We now use the above to define a $\cC_n$-module category structure on $\lmod H$.

\begin{Proposition}\label{prop:Taftaction}
    The category $\lmod H$ is a $\cC_n$-module category
    such that for all $M \in \lmod H$:
    \begin{enumerate}
        \item \label{item:P-locally-free} $P\lact M$ is locally free for all $P \in \proj(\cC_n)$; and
        \item \label{item:S-isomorphic} for all $i \in Q_0$, $(S\lact M)_i \cong M_i$ for all simple $S \in \cC_n$.
    \end{enumerate}
\end{Proposition}
\begin{proof}
    Let us begin by defining the $\cC_n$-module category structure.
    We use the description of $H$-modules from Remark \ref{rem:H-mod} to define the $\cC_n$-module category structure on $\lmod H$ as follows.
    
    Let $\cF$ denote the forgetful functor 
    \[
    \cF: \cC_n = \lmod{T_n(q)} \to \lmod \Bbbk =\vec_\Bbbk.
    \]
    Note that $\lmod{H_i}$ is a $\cC_n$-module category for every $i\in Q_0$: 
    for $H_i = B(n)$ this is defined in Lemma~\ref{lem:truncatedpoly-exact}; 
    for $H_i = \Bbbk$ this is induced by the forgetful functor $\cF$ (see Remark \ref{rmk:two-natural-comodule-algebra}).
    Given $X \in \cC_n$, we shall abuse notation and use $X\lact -$ to denote the corresponding endofunctor of $\lmod{H_i}$ in both cases.
    Importantly, for $M_i \in \lmod{H_i}$ with $i\in R_0$, that is $H_i = B(n)$, note that we have
    equalities (not just isomorphisms) of vector spaces
    \[
    \cF(X\lact M_i) = X\lact \cF(M_i) = X\otimes_\Bbbk\cF(M_i).
    \]
    
    Given $X \in \cC_n$ and $M \in \lmod H$, we define $X \lact M$ to be the $H$-module given by
    \[
    (X\lact M)_i \coloneqq X \lact M_i, \quad (X\lact M)_\alpha \coloneqq X\lact M_\alpha.
    \]
    for all $i\in Q_0$ and $\alpha\in Q_1$.
    Note that $X\lact M_\alpha$ is viewed in $\vec_\Bbbk$ unless $\alpha\in R_1$, in which case $X\lact M_\alpha$ is viewed in $\lmod{B(n)}$.
    From this it follows immediately that $X\lact -: \lmod H \to \lmod H$ and $- \lact M: \cC_n \to \lmod H$ are both exact functors (since kernels and cokernels are computed component-wise), hence defines a $\cC_n$-module category structure on $\lmod H$.

    Now \ref{item:P-locally-free} follows from the definition of the module category structure on $\lmod{H}$ and Lemma~\ref{lem:truncatedpoly-exact}\ref{item:P-action-truncatedpoly}, where we obtain that $(P\lact M)_i = P\lact M_i$ is free for any $P \in \proj(\cC_n)$ (when $i \notin R_0$, $H_i=\Bbbk$ and so all modules are free).
    Similarly, for \ref{item:S-isomorphic} we have $(S\lact M)_i = S \lact M_i \cong M_i$ for any simple $S \in \cC_n$ by Lemma~\ref{lem:truncatedpoly-exact}\ref{item:S-action-truncatedpoly} (when $i \notin R_0$, $H_i=\Bbbk$ and the action of simples is simply $S \lact M_i = \Bbbk \otimes_\Bbbk M_i \cong M_i$.)
\end{proof}

\begin{Remark}
    The module category structure on $\lmod{H}$ over $\cC_n=\lmod{T_n(q)}$ above can also be obtained from the following coaction of $T_n(q)$ on $H$ (cf.\ Lemma \ref{lem:truncatedpoly-exact}):
    \[
    \lambda(e_i) = 1\otimes e_i, \quad \lambda(\varepsilon_i) = x\otimes e_i +g\otimes \varepsilon_i, \quad \lambda(\alpha)=1\otimes \alpha,
    \]
    for all $i\in Q_0$ and $\alpha \in Q_1$.
    Nonetheless, we find the categorical perspective of Proposition \ref{prop:Taftaction} conceptually clearer for the properties desired. 
\end{Remark}

The following example illustrates how simple modules in $\cC_n$ act on $H$-modules:
\begin{example}
    Take $H = \dot{H}(Q,R_0,n)$ with $Q\colon 1 \xleftarrow{\alpha} 2$, $R_0 = \{1,2\}$ and $n=3$, thus $q\in \Bbbk$ denotes a primitive $3$rd root of unity.
    Consider a locally free $H$-module $M$ of the form
    \[\begin{tikzcd}[ampersand replacement = \&]
        M\colon \&\&
        \Bbbk^3 
        \ar[loop left,"{\begin{psmallmatrix} 0 & 0 & 0 \\ 1 & 0 & 0 \\ 0 & 1 & 0 \end{psmallmatrix}}"] 
        \&\&\&
        \Bbbk^3
        \ar[lll,"{\begin{psmallmatrix} x_{11} & x_{12} & x_{13} \\ x_{21} & x_{22} & x_{23} \\ x_{31} & x_{32} & x_{33} \end{psmallmatrix}}"']
        \ar[loop right,"{\begin{psmallmatrix} 0 & 0 & 0 \\ 1 & 0 & 0 \\ 0 & 1 & 0 \end{psmallmatrix}}"] 
    \end{tikzcd}\]
    with $x_{ij}\in\Bbbk$ for all $1\leq i,j \leq 3$.
    We have the following canonical isomorphism $S_1\lact M \cong N$ of $H$-modules:
    \[\begin{tikzcd}[ampersand replacement = \&]
        S_1\lact M\colon \&\&
        \Bbbk^3 
        \ar[loop left,"{\begin{psmallmatrix} 0 & 0 & 0 \\ q^2 & 0 & 0 \\ 0 & q^2 & 0 \end{psmallmatrix}}"] 
        \ar[ddd,"{\begin{psmallmatrix} 1 & 0 & 0 \\ 0 & q & 0 \\ 0 & 0 & q^2 \end{psmallmatrix}}"']
        \&\&\&
        \Bbbk^3
        \ar[lll,"{\begin{psmallmatrix} x_{11} & x_{12} & x_{13} \\ x_{21} & x_{22} & x_{23} \\ x_{31} & x_{32} & x_{33} \end{psmallmatrix}}"']
        \ar[loop right,"{\begin{psmallmatrix} 0 & 0 & 0 \\ q^2 & 0 & 0 \\ 0 & q^2 & 0 \end{psmallmatrix}}"] 
        \ar[ddd,"{\begin{psmallmatrix} 1 & 0 & 0 \\ 0 & q & 0 \\ 0 & 0 & q^2 \end{psmallmatrix}}"]
        \\ \\ \\ 
        N\colon \&\&
        \Bbbk^3 
        \ar[loop left,"{\begin{psmallmatrix} 0 & 0 & 0 \\ 1 & 0 & 0 \\ 0 & 1 & 0 \end{psmallmatrix}}"] 
        \&\&\&
        \Bbbk^3
        \ar[lll,"{\begin{psmallmatrix} x_{11} & q^2 x_{12} & q x_{13} \\ q x_{21} & x_{22} & q^2 x_{23} \\ q^2 x_{31} & q x_{32} & x_{33} \end{psmallmatrix}}"']
        \ar[loop right,"{\begin{psmallmatrix} 0 & 0 & 0 \\ 1 & 0 & 0 \\ 0 & 1 & 0 \end{psmallmatrix}}"] 
    \end{tikzcd}\]
\end{example}

\begin{Remark}
    To the best of our knowledge, (co)actions of Taft algebras on finite dimensional algebras have mostly been focusing on classifying the (co)actions; see e.g.\ \cite{KW_Hopf-path} for path algebras, \cite{EO_FTC,Mombelli_uqsl2} for self-injective algebras, and \cite{CY-Taft_action_fda} for some general results.
    The point of our work is to show that these actions can be used to study the representation theory of the algebras (co)acted upon.
\end{Remark}

\section{Proof of the main theorem}

From here on we fix the $\cC_n$-module category structure on $\cM\coloneqq \lmod H$ as in Proposition \ref{prop:Taftaction}.
Recall the subcategory $\cP\cM \subseteq \cM$ defined in Definition \ref{defn:PM}.
We can now state our main theorem.
\begin{Theorem}\label{thm:mainthm}
    Let $\cM \coloneqq  \lmod H$.
    For $M \in \cM$ a rigid module, $P\lact M$ contains $M$ as a summand for any non-zero $P \in \proj(\cC_n)$.
    In particular, all rigid $H$-modules are in $\cP\cM$, and hence are $\tau$-locally free.
\end{Theorem}

Before we prove our theorem, we need two lemmas.
\begin{Lemma}\label{lem:orbit-vanishing-Ext-splits}
    Let $\cM$ be a $\cC$-module category for a finite tensor category $\cC$, and let $M \in \cM$ and $X \in \cC$.
    Denote by $JH(X)$ the multiset of simple objects appearing in the Jordan--H\"older filtration of $X$, so that $JH(X)$ has number of elements equal to the length of $X$.
    Suppose that $\Ext^1_\cM(S\lact M,S'\lact M) =0$ for all $S,S' \in JH(X)$.
    Then
    \[
    X \lact M \cong \bigoplus_{S \in JH(X)} S\lact M.
    \]
\end{Lemma}
\begin{proof}
    We induct on the length of $X$.
    The base case of length $1$ is immediate.
    Now suppose $X$ has length at least $2$, and consider $S' \subsetneq X$ a simple submodule of $X$, so that we get an exact sequence
    \[
    0 \to S' \to X \to X' \to 0.
    \]
    By construction, $JH(X)= JH(X') \sqcup \{S'\}$.
    In particular, the length of $X'$ is strictly less than that of $X$, and $JH(X') \subset JH(X)$, so that we can apply the induction hypothesis to obtain
    \[
    X'\lact M \cong \bigoplus_{S'' \in JH(X')} S''\lact M.
    \]
    By exactness of $-\lact M$, the exact sequence above gives
    \[
    0 \to S'\lact M \to X\lact M \to X'\lact M \to 0.
    \]
    Replacing $X'\lact M$ with the direct sum above and applying the vanishing of $\Ext^1_{\cM}(S\lact M, S'\lact M)$, we get that $X\lact M \cong \bigoplus_{S \in JH(X)} S\lact M$ as required.
\end{proof}

\begin{Lemma}\label{lem:simple_action_on_rigids}
    Let $M \in \lmod{H}$ be a rigid module.
    Then $S_\ell\lact M \cong M$ for all $0\leq \ell < n$.
\end{Lemma}
\begin{proof}
    For a dimension vector ${\bf d} = (d_i)_{i\in Q_0} \in \mathbb{N}^{Q_0}$, consider the affine variety
    \begin{align*}
        \rep(H,{\bf d}) \subseteq \prod_{\alpha\in  Q^\varepsilon_1} \Bbbk^{d_{t(\alpha)}\times d_{s(\alpha)}}
    \end{align*}
    whose elements $M = (M_\alpha)_{\alpha\in Q^\varepsilon_1} \in \rep(H,{\bf d})$ satisfy the matrix equations
    \begin{align*}
        M_{\varepsilon_i} = M_{\varepsilon_j}^n = M_{\varepsilon_{t(\alpha)}}M_{\alpha} - M_{\alpha}M_{\varepsilon_{s(\alpha)}} = 0
    \end{align*}
    for all $i\in Q_0\setminus R_0$, $j\in R_0$ and $\alpha\in R_1$.
    Note that every element of $\rep(H,{\bf d})$ defines an $H$-module with dimension vector ${\bf d}$, and every $H$-module with dimension vector ${\bf d}$ is isomorphic to a module arising this way. By some abuse of notation, we identify points of $\rep(H,{\bf d})$ and $H$-modules with dimension vector ${\bf d}$.
    The group $\GL({\bf d}) \coloneqq \prod_{i\in Q_0} \GL(d_i)$ acts on $\rep(H,{\bf d})$ via conjugation such that $\GL({\bf d})$-orbits correspond to isomorphism classes of $H$-modules. Given $M\in\rep(H,{\bf d})$, write $\mathcal{O}(M) \coloneqq \GL({\bf d}).M$ for the $\GL({\bf d})$-orbit of $M$ in $\rep(H,{\bf d})$. 
    Set $\bar{H} := \prod_{i\in Q_0} H_i$, and consider the canonical projection $\pi\colon \rep(H,{\bf d}) \to \rep(\bar{H},{\bf d})$.
    Following \cite[Section 5.3]{BobinskiSchröer}, we consider the stratification 
    \begin{align*}
        \rep(H,{\bf d}) = \bigsqcup_{\mathcal{O}} \mathcal{S}(\mathcal{O})
        &&\text{ with }&&
        \mathcal{S}(\mathcal{O}) := \bigcup_{\bar{M} \in \mathcal{O}} \pi^{-1}(\bar{M}) 
    \end{align*}
    where the disjoint union is taken over all $\GL({\bf d})$-orbits $\mathcal{O}$ in $\rep(\bar{H},{\bf d})$. It is routine to see that $\mathcal{S}(\mathcal{O})$
    is an irreducible locally closed subset of $\rep(H,{\bf d})$ for every $\GL({\bf d})$-orbit $\mathcal{O}$ in $\rep(\bar{H},{\bf d})$.

    Let $M \in \lmod{H}$ be a rigid module.
    Set ${\bf d} := (\dim_\Bbbk M_i)_{i\in Q_0}$ and $\mathcal{O} \coloneqq \mathcal{O}(\pi(M)) \subseteq \rep(\bar{H},{\bf d})$.
    By Proposition~\ref{prop:Taftaction}\ref{item:S-isomorphic}, we have that $M' \coloneqq S_\ell\lact M \in \mathcal{S}(\mathcal{O})$.
    Since $M$ is rigid and $S_\ell\lact -$ is an autoequivalence, $M'$ is rigid.
    By Voigt's Lemma both orbits $\mathcal{O}(M)$ and $\mathcal{O}(M')$ are open subsets of $\rep(H,{\bf d})$.
    Thus, their non-empty intersections with $\mathcal{S}(\mathcal{O})$ are two open subsets of $\mathcal{S}(\mathcal{O})$. 
    By irreducibility of the stratum $\mathcal{S}(\mathcal{O})$, the orbits $\mathcal{O}(M)$ and $\mathcal{O}(M')$ must intersect non-trivially. This proves $M'\cong M$.
\end{proof}

\begin{proof}[Proof of Theorem~\ref{thm:mainthm}]
    Let $M \in \cM = \lmod H$ be a rigid module.
    By Lemma~\ref{lem:simple_action_on_rigids}, we have $S_\ell \lact M \cong M$ for all $0\leq \ell < n$.
    In particular, $\Ext^1_\cM(S_\ell \lact M , S_{\ell'} \lact M) \cong \Ext^1(M,M) = 0$ for all $0\leq \ell,\ell' < n$.
    Take any non-zero $P \in \proj(\cC_n)$.
    By Lemma~\ref{lem:orbit-vanishing-Ext-splits}, we get that $P \lact M \cong M^{\oplus \dim_\Bbbk (P)}$, hence $M \in \cP\cM$ by definition.
    
    By Proposition~\ref{prop:Taftaction}, we get that $P\lact M$ is locally free, hence its summand $M$ is also locally free (note that every projective $H_i$-module is free).
    Combined with Theorem~\ref{thm:ARtranslate-closed}, we get that $M$ is moreover $\tau$-locally free, as required.
\end{proof}

The following example shows that $\cP\cM$ does not always contain all rigid modules for a general $\cC$-module category $\cM$.
\begin{Example}\label{eg:rigidnotinPM}
    Consider $\cC = \lmod{T_2(q)} = \cM$.
    Then it is known that all indecomposable modules are rigid: $T_2(q)$ is just the preprojective algebra of type $A_2$. 
    In particular, the simple modules -- which are not projective -- are rigid.
    However, we have that $\cP\cM=\proj(\cM)$; this just says that $\cC$ is an exact $\cC$-module category, which is true for all finite tensor categories $\cC$. 
    Indeed, for any $X \in \cC$ the functor $-\lact X = -\otimes X$ has an exact right adjoint $-\otimes \leftdual{X}$, hence sends projectives to projectives. 
\end{Example}

\printbibliography

@incollection {Ringel_ex-seq-Artin,
    AUTHOR = {Ringel, Claus Michael},
     TITLE = {The braid group action on the set of exceptional sequences of
              a hereditary {A}rtin algebra},
 BOOKTITLE = {Abelian group theory and related topics ({O}berwolfach, 1993)},
    SERIES = {Contemp. Math.},
    VOLUME = {171},
     PAGES = {339--352},
 PUBLISHER = {Amer. Math. Soc., Providence, RI},
      YEAR = {1994},
      ISBN = {0-8218-5178-0},
   MRCLASS = {16E60 (16G20 16G70 20F36)},
  MRNUMBER = {1293154},
MRREVIEWER = {Dieter\ Happel},
       DOI = {10.1090/conm/171/01786},
       URL = {https://doi.org/10.1090/conm/171/01786},
}

@inproceedings {CB_ex-seq-quiver,
    AUTHOR = {Crawley-Boevey, William},
     TITLE = {Exceptional sequences of representations of quivers},
 BOOKTITLE = {Proceedings of the {S}ixth {I}nternational {C}onference on
              {R}epresentations of {A}lgebras ({O}ttawa, {ON}, 1992)},
    SERIES = {Carleton-Ottawa Math. Lecture Note Ser.},
    VOLUME = {14},
     PAGES = {7},
 PUBLISHER = {Carleton Univ., Ottawa, ON},
      YEAR = {1992},
   MRCLASS = {16G20},
  MRNUMBER = {1206935},
MRREVIEWER = {Peter\ Dr\"axler},
}

@article {CY-Taft_action_fda,
    AUTHOR = {Centrone, Lucio and Yasumura, Felipe},
     TITLE = {Actions of {T}aft's algebras on finite dimensional algebras},
   JOURNAL = {J. Algebra},
  FJOURNAL = {Journal of Algebra},
    VOLUME = {560},
      YEAR = {2020},
     PAGES = {725--744},
      ISSN = {0021-8693,1090-266X},
   MRCLASS = {16T05 (16W50)},
  MRNUMBER = {4111735},
MRREVIEWER = {Gast\'on\ Andr\'es\ Garc\'ia},
       DOI = {10.1016/j.jalgebra.2020.06.007},
       URL = {https://doi.org/10.1016/j.jalgebra.2020.06.007},
}

@article {KW_Hopf-path,
    AUTHOR = {Kinser, Ryan and Walton, Chelsea},
     TITLE = {Actions of some pointed {H}opf algebras on path algebras of
              quivers},
   JOURNAL = {Algebra Number Theory},
  FJOURNAL = {Algebra \& Number Theory},
    VOLUME = {10},
      YEAR = {2016},
    NUMBER = {1},
     PAGES = {117--154},
      ISSN = {1937-0652,1944-7833},
   MRCLASS = {16T05 (05C20 16G20)},
  MRNUMBER = {3474843},
MRREVIEWER = {Sarah\ J.\ Witherspoon},
       DOI = {10.2140/ant.2016.10.117},
       URL = {https://doi.org/10.2140/ant.2016.10.117},
}

@misc{CH_eqAuscor-findim,
      title={A subcategory with remarkable homological properties via actions of tensor categories}, 
      author={Kevin Coulembier and Edmund Heng},
      year={2026},
      eprint={2610.17646},
      archivePrefix={arXiv},
      primaryClass={math.RT},
      url={https://arxiv.org/abs/2610.17646}, 
}

@article {Kulshammer_prospecies,
    AUTHOR = {K\"ulshammer, Julian},
     TITLE = {Pro-species of algebras {I}: {B}asic properties},
   JOURNAL = {Algebr. Represent. Theory},
  FJOURNAL = {Algebras and Representation Theory},
    VOLUME = {20},
      YEAR = {2017},
    NUMBER = {5},
     PAGES = {1215--1238},
      ISSN = {1386-923X,1572-9079},
   MRCLASS = {16G10 (16G20)},
  MRNUMBER = {3707912},
MRREVIEWER = {Nadiya\ Gubareni},
       DOI = {10.1007/s10468-017-9683-2},
       URL = {https://doi.org/10.1007/s10468-017-9683-2},
}

@article {AM_modulecatHopf,
    AUTHOR = {Andruskiewitsch, Nicol\'as and Mombelli, Juan Mart\'in},
     TITLE = {On module categories over finite-dimensional {H}opf algebras},
   JOURNAL = {J. Algebra},
  FJOURNAL = {Journal of Algebra},
    VOLUME = {314},
      YEAR = {2007},
    NUMBER = {1},
     PAGES = {383--418},
      ISSN = {0021-8693,1090-266X},
   MRCLASS = {16W30},
  MRNUMBER = {2331768},
MRREVIEWER = {\c Serban\ Raianu},
       DOI = {10.1016/j.jalgebra.2007.04.006},
       URL = {https://doi.org/10.1016/j.jalgebra.2007.04.006},
}

@article {EO_FTC,
    AUTHOR = {Etingof, Pavel and Ostrik, Viktor},
     TITLE = {Finite tensor categories},
   JOURNAL = {Mosc. Math. J.},
  FJOURNAL = {Moscow Mathematical Journal},
    VOLUME = {4},
      YEAR = {2004},
    NUMBER = {3},
     PAGES = {627--654, 782--783},
      ISSN = {1609-3321,1609-4514},
   MRCLASS = {18D10 (16W30)},
  MRNUMBER = {2119143},
MRREVIEWER = {Sebastian\ Burciu},
       DOI = {10.17323/1609-4514-2004-4-3-627-654},
       URL = {https://doi.org/10.17323/1609-4514-2004-4-3-627-654},
}

@article {Mombelli_uqsl2,
    AUTHOR = {Mombelli, Mart\'in},
     TITLE = {Module categories over pointed {H}opf algebras},
   JOURNAL = {Math. Z.},
  FJOURNAL = {Mathematische Zeitschrift},
    VOLUME = {266},
      YEAR = {2010},
    NUMBER = {2},
     PAGES = {319--344},
      ISSN = {0025-5874,1432-1823},
   MRCLASS = {16T05 (18D10 19D23)},
  MRNUMBER = {2678630},
MRREVIEWER = {V.\ A.\ Artamonov},
       DOI = {10.1007/s00209-009-0571-2},
       URL = {https://doi.org/10.1007/s00209-009-0571-2},
}

@article {GLS_rigidSchur,
    AUTHOR = {Gei\ss, Christof and Leclerc, Bernard and Schr\"oer, Jan},
     TITLE = {Rigid modules and {S}chur roots},
   JOURNAL = {Math. Z.},
  FJOURNAL = {Mathematische Zeitschrift},
    VOLUME = {295},
      YEAR = {2020},
    NUMBER = {3-4},
     PAGES = {1245--1277},
      ISSN = {0025-5874,1432-1823},
   MRCLASS = {16G20 (16D90 16E60)},
  MRNUMBER = {4125687},
MRREVIEWER = {Alireza\ Nasr-Isfahani},
       DOI = {10.1007/s00209-019-02396-5},
       URL = {https://doi.org/10.1007/s00209-019-02396-5},
}

@article {GLS_quiversI,
    AUTHOR = {Geiss, Christof and Leclerc, Bernard and Schr\"oer, Jan},
     TITLE = {Quivers with relations for symmetrizable {C}artan matrices
              {I}: {F}oundations},
   JOURNAL = {Invent. Math.},
  FJOURNAL = {Inventiones Mathematicae},
    VOLUME = {209},
      YEAR = {2017},
    NUMBER = {1},
     PAGES = {61--158},
      ISSN = {0020-9910,1432-1297},
   MRCLASS = {16G10 (16G20 16G70)},
  MRNUMBER = {3660306},
MRREVIEWER = {Intan\ Muchtadi-Alamsyah},
       DOI = {10.1007/s00222-016-0705-1},
       URL = {https://doi.org/10.1007/s00222-016-0705-1},
}

@book {EGNO,
    AUTHOR = {Etingof, Pavel and Gelaki, Shlomo and Nikshych, Dmitri and
              Ostrik, Victor},
     TITLE = {Tensor categories},
    SERIES = {Mathematical Surveys and Monographs},
    VOLUME = {205},
 PUBLISHER = {American Mathematical Society, Providence, RI},
      YEAR = {2015},
     PAGES = {xvi+343},
      ISBN = {978-1-4704-2024-6},
   MRCLASS = {18D10 (16T05)},
  MRNUMBER = {3242743},
MRREVIEWER = {Julien\ Bichon},
       DOI = {10.1090/surv/205},
       URL = {https://doi.org/10.1090/surv/205},
}

@article {BobinskiSchröer,
    AUTHOR = {Bobi\'nski, Grzegorz and Schr\"oer, Jan},
     TITLE = {Algebras with irreducible module varieties {I}},
   JOURNAL = {Adv. Math.},
  FJOURNAL = {Advances in Mathematics},
    VOLUME = {343},
      YEAR = {2019},
     PAGES = {624--639},
      ISSN = {0001-8708,1090-2082},
   MRCLASS = {16G10 (16G20)},
  MRNUMBER = {3883216},
MRREVIEWER = {Wolfgang\ Rump},
       DOI = {10.1016/j.aim.2018.11.024},
       URL = {https://doi.org/10.1016/j.aim.2018.11.024},
}

@article {HuberyKrause,
    AUTHOR = {Hubery, Andrew and Krause, Henning},
     TITLE = {A categorification of non-crossing partitions},
   JOURNAL = {J. Eur. Math. Soc. (JEMS)},
  FJOURNAL = {Journal of the European Mathematical Society (JEMS)},
    VOLUME = {18},
      YEAR = {2016},
    NUMBER = {10},
     PAGES = {2273--2313},
      ISSN = {1435-9855,1435-9863},
   MRCLASS = {17B22 (05A17 05E15 16E20 16E60 20F55)},
  MRNUMBER = {3551191},
MRREVIEWER = {Volodymyr\ Mazorchuk},
       DOI = {10.4171/JEMS/641},
       URL = {https://doi.org/10.4171/JEMS/641},
}

@article {Pfeifer,
    AUTHOR = {Pfeifer, Calvin},
     TITLE = {A generic classification of locally free representations of
              affine {GLS} algebras},
   JOURNAL = {J. Algebra},
  FJOURNAL = {Journal of Algebra},
    VOLUME = {664},
      YEAR = {2025},
     PAGES = {414--467},
      ISSN = {0021-8693,1090-266X},
   MRCLASS = {16G10 (14L30 16G20 16G60 16G70)},
  MRNUMBER = {4813554},
MRREVIEWER = {Zhonghua\ Zhao},
       DOI = {10.1016/j.jalgebra.2024.09.013},
       URL = {https://doi.org/10.1016/j.jalgebra.2024.09.013},
}

\end{document}